\documentclass[12pt]{article}

\usepackage{amsmath}
\usepackage{amssymb}
\usepackage{amsthm}
\usepackage[all,cmtip]{xy}
\usepackage{fancyhdr}
\usepackage{euscript}
\usepackage{setspace}
\usepackage{graphicx,epsfig}
\usepackage{palatino, url, multicol}
\usepackage{mathrsfs} 
\usepackage{enumerate}   
\usepackage{verbatim} 

\theoremstyle{definition}

\newtheorem{theorem}{Theorem}[section]

\newtheorem{proposition}[theorem]{Proposition}

\newtheorem{corollary}[theorem]{Corollary}

\newcommand{\IR}{\hbox{$\mathbb{R}$}}
\newcommand{\IN}{\hbox{$\mathbb{N}$}}
\newcommand{\IZ}{\hbox{$\mathbb{Z}$}}

\newcommand{\IQ}{\hbox{$\mathbb{Q}$}}

\def\it{\itshape}

\def\tt{\texttt}
\def\bf{\textbf}

\def\IR{{\mathbb{R}}}

\def\IZ{{\mathbb{Z}}}

\def\IQ{{\mathbb{Q}}}
\def\IN{{\mathbb{N}}}

\def\I1{{\mathbb{1}}}

\def\limx0{\lim_{x \to 0}}

\def\intxyleq1{\underset{\| x - y  \| \leq 1}{\int}}
\def\intxygeq1{\underset{\| x - y  \| \geq 1}{\int}}
\def\intxizetaleq1{\underset{\| \xi - \zeta  \| \leq 1}{\int}}
\def\intxizetageq1{\underset{\| \xi - \zeta \| \geq 1}{\int}}

\def\tab{\hskip 1mm}

\def\tab{\hspace{.1pc}}
\def\ttab{\hspace{1pc}}

\newcounter{hours}
\newcounter{minutes}
\newcommand\printtime{%
  \setcounter{hours}{\the\time/60}%
  \setcounter{minutes}{\the\time-\value{hours}*60}%
  \ifthenelse{\value{hours} > 12}
     {
       \setcounter{hours}{\value{hours}-12}%
       \thehours:\theminutes \ p.m.                
     }
     {
       \thehours:\theminutes \ a.m.                
     } 
}

\def\putdate{{\tt Compiled on \the\month-\the\day-\the\year \ at\printtime} \\}

\usepackage[linesnumbered,ruled]{algorithm2e}
\usepackage{enumitem,color,amssymb}

\begin{document} 

\title{The Ostrowski Expansions Revealed}
 \author{A. Bourla}
 \date{\today}
 \maketitle
\begin{abstract}
\noindent We provide algorithms for the absolute and alternating Ostrowski Expansions of the continuum and provide proofs for their uniqueness.  
\end{abstract} 

\section{Introduction}{}

\noindent  

\noindent The various algorithms that construe the Ostrowski Expansions rely on the continued fraction expansion of a fixed irrational number $\alpha$ in the interval in order to represent other real numbers `base--$\alpha$' and `base--$(-\alpha)$'. They are utilized in a broad range of applications, ranging from Diophantine Approximation \cite{B1, IN} to symbolic dynamics and coding theory \cite{R,S}, for a through survey refer to \cite{B2}. Given $r \in \IR$, we define the \bf{floor} $\lfloor r \rfloor$ of the real number $r$ to be the largest integer smaller than or equal to $r$ and expand $r$ as a continued fraction using the following iteration scheme:\\  

\IncMargin{1em}
\begin{algorithm}[H]
\SetKwInOut{Input}{input}\SetKwInOut{Output}{output}
\Input{$r \in \IR$}
\Output{$\ell \in \IN^\infty := \IZ_{\ge1}\cup \{\infty\}$, \tab $a_0 \in \IZ, \tab \left<a_k\right>_1^\ell \subset \IZ_{\ge1}$}
set $a_0 := \lfloor r \rfloor, \tab \alpha_0 := r-a_0, \tab \ell := \infty, \tab k:=1$\;
\While{$\alpha_{k-1} > 0$}{
set $a_k:=\lfloor1/\alpha_{k-1}\rfloor$\;
set $\alpha_k:=1/\alpha_{k-1}-a_k \in [0,1)$\;
set $k:=k+1$\; 
}
set $\ell:= k-1$\;
\caption{continued fraction expansion}\label{rcf}
\end{algorithm}\DecMargin{1em}
\vspace{1pc}

\noindent The proof of the existence and uniqueness for this expansion as well as the assertion of the rest of the claims made in this section can be found in the classical exposition \cite{K}. This iteration process will terminate with a finite value $\ell$ precisely when $\alpha$ is rational. The assignment of the digit $a_k$ in line--3 yields the inequality
\begin{equation}\label{a_k}
a_k\alpha_{k-1} \le 1 < (a_k+1)\alpha_{k-1}, \ttab 1 \le k < \ell+1,
\end{equation}
(where $\infty+1 := \infty = \ell$ when applicable). After we rewrite the assignment in line--4 as $\alpha_{k-1} = (a_k + \alpha_k)^{-1}$, we obtain the expansion
\[r = a_0 + \alpha_0 = a_0 + \dfrac{1}{a_1+\alpha_1} = a_0 + \dfrac{1}{a_1+\dfrac{1}{a_2+\alpha_2}} = ... = a_0 + \dfrac{1}{a_1 + \dfrac{1}{a_2 + \dfrac{1}{a_3+\dfrac{1}{\ddots}}}},\]
whose truncation at the $k < \ell+1$ step yields the convergent
\[\frac{p_k}{q_k} := a_0 + \dfrac{1}{a_1 + \dfrac{1}{a_2 +\dfrac{1}{\ddots \begin{matrix}\\ +\dfrac{1}{a_k} \end{matrix}}}}.\]
\noindent We fix $\alpha \in (0,1) \backslash \IQ$ throughout and, after we plug it as input in the algorithm \ref{rcf}, we obtain the value $a_0=0$ and the infinite digit sequence $\left<a_k\right>_1^\infty$. We end this section by quoting two well known facts about the resulting sequence of convergents $\left<p_k/q_k\right>_0^\infty$, namely the recursion equation
\begin{equation}\label{pq_recursion}
q_{-1}=p_0 :=0, \ttab p_{-1}=q_0 :=1, \ttab p_k =  a_k{p_{k-1}} + p_{k-2},\ttab q_k =  a_k{q_{k-1}} + q_{k-2}, \ttab k \ge 1
\end{equation}
and the inequality 
\begin{equation}\label{convergent_bound}
\left|\alpha - \dfrac{p_k}{q_k} \right| < \dfrac{1}{q_k^2}, \ttab k \ge 0.
\end{equation}

\section{Basic definitions and identities}
 
The base--$\alpha$ and base--$(-\alpha)$ Ostrowski Expansions are dot products of two sequences: a digit sequence and a sequence of certain coefficients depending on $\alpha$, which will we now define and study. After applying algorithm \ref{rcf} and letting $\left<p_k/q_k\right>_{-1}^\infty$ be as in equation \eqref{pq_recursion}, we define the coefficients
\begin{equation}\label{theta_def}
\theta_k := q_k\alpha-p_k, \ttab k \ge -1.
\end{equation}
Utilizing this definition and the equations \eqref{pq_recursion} in an induction argument, we arrive at the recursion equations
\begin{equation}\label{theta_recursion}
\theta_{-1} = -1, \ttab \theta_0 = \alpha, \ttab \theta_k = \theta_{k-2} +  a_k\theta_{k-1}, \ttab |\theta_k| = |\theta_{k-2}| - a_k|\theta_{k-1}|, \ttab k \ge 1.
\end{equation}
Multiplying both sides of the inequality \eqref{convergent_bound} by $q_k$ yields the inequality
\begin{equation}\label{theta_bound}
|\theta_k| < q_k^{-1}, \ttab k \ge 0.
\end{equation}
Assuming that $\theta_k = -\theta_{k-1}\alpha_k$ for some $k \in \IN$, where $\alpha_k$ is as in line--4 of algorithm \ref{rcf}, we use this recursion relationship to obtain the equality
\[\theta_{k+1} = \theta_{k-1} + a_{k+1}\theta_k = \theta_{k-1}\left(1-a_{k+1}\alpha_k \right) = \theta_{k-1}\alpha_k\left(\dfrac{1}{\alpha_k}-a_{k+1} \right) = -\theta_k{\alpha_{k+1}}. \]
Since we have $\theta_0=\alpha=-\theta_{-1}\alpha$, that is, the assumption holds for $k=0$, we have just proved by induction that
\begin{equation}\label{theta_induction}
\theta_k = -\theta_{k-1}\alpha_k = (-1)^k{\alpha_0}{\alpha_1}...{\alpha_k}, \ttab k \ge 0,
\end{equation}
In particular, this shows that the sequence $\left<\theta_k\right>_{-1}^\infty$ is alternating as in
\begin{equation}\label{theta_alt}
|\theta_k| = (-1)^k\theta_k, \ttab k \ge -1.
\end{equation}
In tandem with the inequality \eqref{a_k}, we obtain the inequality
\[\dfrac{a_{k+1}|\theta_k|}{|\theta_{k-1}|} = a_{k+1}\alpha_k < 1 < (a_{k+1}+1)\alpha_k = \dfrac{(a_{k+1}+1)|\theta_k|}{|\theta_{k-1}
|} ,\]
that is,
\begin{equation}\label{theta_a_k}
a_{k+1}|\theta_k| < |\theta_{k-1}| < (a_{k+1}+1)|\theta_k|, \ttab k \ge 0.
\end{equation}
Since by definition \eqref{theta_recursion}, we have $|\theta_0| < |\theta_{-1}|$ and since $|\alpha_k|<1$ by its definition in line--4 of algorithm \ref{rcf}, the formula \eqref{theta_induction} also asserts that the sequence $\left<|\theta_k|\right>_{-1}^\infty$ is strictly decreasing to zero, that is,
\begin{equation}\label{eta_decreasing}
|\theta_{k+1}| < |\theta_k| \to 0, \ttab \text{as \tab $k \to \infty$}.
\end{equation}
While, by the alternating series test, this is enough to assert the convergence of the series $\sum_{k:=1}^\infty{a_k}\theta_{k-1}$, this series, in fact, converges absolutely:
\begin{proposition}\label{conv}
The infinite series $\sum_{k:=1}^\infty{a_k}|\theta_{k-1}|$ converges for all $\alpha$.
\end{proposition}

\begin{proof}
By the inequality \eqref{theta_a_k}, we see that $\sum_{k=1}^\infty{a_k}|\theta_{k-1}| < \sum_{k=-1}^\infty|\theta_k|$ and by the recursion equation \eqref{pq_recursion} and the inequality \eqref{theta_bound} we have
\[\dfrac{|\theta_{k+1}|}{|\theta_k|} < \dfrac{q_k}{q_{k+1}} = \dfrac{q_k}{a_{k+1}q_k + q_{k-1}} < \dfrac{1}{a_{k+1}}, \ttab k \ge 0.\]
Thus, as long as $\displaystyle{\limsup_{k\to\infty}\left<a_k\right>} \ge 2$, we conclude convergence from the simple comparison and ratio tests. When $\displaystyle{\limsup_{k\to\infty}\left<a_k\right>} = 1$, then $\alpha$ must be a noble number, whose continued fraction expansion ends with a tail of 1's. By the limit comparison test, we need only establish the convergence for this tail, that is for $\alpha$ where $a_k=1$ for all $k \ge 1$. After using the assignments of line--3 and line--4 in algorithm \ref{rcf}, we write  
\[\alpha_k = \dfrac{1}{1+\alpha_{k+1}} = \dfrac{1}{1 + \dfrac{1}{1+\dfrac{1}{\ddots}}} = \dfrac{1}{1+\alpha_k}, \ttab k \ge 0.\]
The solution for the resulting quadratic equation is the golden section $\phi := \alpha_k = .5(-1+5^{.5}) \approx .618$. Then the identity \eqref{theta_induction} and the geometric sum formula assert that 
\[\sum_{k=1}^\infty{a_k}|\theta_{k-1}| = \sum_{k=0}^\infty|\theta_k| = \sum_{k=0}^\infty\alpha_0\alpha_1...\alpha_k =
\sum_{k=0}^\infty\phi^{k+1} = 1 + \phi,\]
which proves convergence for this case as well.
\end{proof} 

\noindent After fixing a finite index $n \ge 1$, we can now use formula \eqref{theta_recursion} to rewrite the tail\\ $\sum_{k:=n}^\infty{a_k}|\theta_{k-1}|$ as the telescoping series 
\begin{equation}\label{eta_tail}
a_n|\theta_{n-1}| +  a_{n+1}|\theta_n| + ...  = (|\theta_{n-2}| - |\theta_n|) + (|\theta_{n-1}| - |\theta_{n+1}|)+ ... = |\theta_{n-2}| + |\theta_{n-1}|.
\end{equation}
After plugging $n:=1$ and plugging the values for $\theta_{-1}$ and $\theta_0$ as in the equation \eqref{theta_recursion}, we use this identity to explicitly evaluate the sum
\begin{equation}\label{abs_rep}
\sum_{k=1}^\infty{a_k}|\theta_{k-1}| = 1 + \alpha.
\end{equation}
We use the relationship \eqref{theta_recursion} again, we can write $\sum_{k:=1}^\infty{a_k}\theta_{k-1}$ as the telescopic series
\[a_1\theta_0 +  a_2\theta_1 + a_3\theta_2 + ...  = (\theta_1 - \theta_{-1}) + (\theta_2 - \theta_0) + (\theta_3 - \theta_1) + ... =  - \theta_{-1} - \theta_0,\]
and evaluate this sum as
\begin{equation}\label{cond_rep}
\sum_{k=1}^\infty{a_k}\theta_{k-1} = 1 - \alpha.
\end{equation}
Subtracting the sum \eqref{cond_rep} from the sum \eqref{abs_rep} and dividing by two yields the self representations
\begin{equation}\label{self_rep}
\alpha = \sum_{k=1}^\infty{a_{2k}}|\theta_{2k-1}| = -\sum_{k=1}^\infty{a_{2k}}\theta_{2k-1}.
\end{equation}
Adding the sum \eqref{cond_rep} to the sum \eqref{abs_rep} and then dividing by two yields the expansion of unity
\begin{equation}\label{unit_rep}
1 = \sum_{k=0}^\infty{a_{2k+1}}\theta_{2k}=\sum_{k=0}^\infty{a_{2k+1}}|\theta_{2k}|.
\end{equation}

\section{The Absolute Ostrowski Expansion}

The base--$\alpha$ Absolute Ostowski Expansion is a sum of the form $\sum_{k=1}^\infty{d_k}|\theta_{k-1}|$, where $0 \le d_k \le a_k$ along with all its finite truncations. While a simple comparison to the convergent series in proposition \ref{conv} proves its existence, it is by no means unique. For instance, using the definition \eqref{theta_recursion} of $\theta_0 := \alpha$ we see that after setting $\ell = d_1 :=1$, we obtain the self-expansion, which is different from formula \eqref{self_rep}. To achieve uniqueness, we will require the digit sequence to adhere to the so called Markov Conditions. We say that the sequence $\left<b_k\right>_1^\infty$ is $\boldsymbol{\alpha}$\bf{--admissible} when: 
\begin{enumerate}[label=(\roman*)]
\item for all $k$ we have $0 \le b_k \le a_k$ not all zero.
\item if $b_k = a_k$, then $b_{k+1} = 0$.
\item for infinitely many odd and even indexes $k$ we have $b_k \le a_k-1$.  
\end{enumerate}
We then expand this definition to the finite digit sequence $\left<b_k\right>_1^\ell$ with $\ell < \infty$ and $b_\ell \ge 1$ by testing these conditions against the infinite sequence $\left<b_k\right>_1^\infty$ obtained by letting $b_k := 0$ for all $k \ge \ell+1$.

\begin{theorem}\label{Ost_abs}
For all $\ell \in \IN^\infty$ and $\alpha$--admissible digit sequences $\left<b_k\right>_1^\ell$, we have $\sum_{k=1}^\ell{b_k|\theta_{k-1}|} \in (0,1)$. Furthermore, for every real number $\beta \in (0,1)$, there exists a unique limit $\ell\in\IN^\infty$ and an $\alpha$--admissible digit sequence $\left<b_k\right>_1^\ell$ (with $b_\ell \ge 1$ when $\ell$ is finite) such that $\beta = \sum_{k=1}^\ell{b_k|\theta_{k-1}|}$.
\end{theorem} 

\begin{proof}
Given a limit $\ell \in \IN^\infty$ and an $\alpha$--admissible digit sequence $\left<b_k\right>_1^\ell$, we first show that $\sum_{k:=1}^\ell{b_k}|\theta_{k-1}| \in (0,1)$. When $\ell=0$ we obtain the vacuous expansion of nullity and when $1 \le \ell < \infty$ we first pad this sequence with a tail of zeros and obtain the $\alpha$--admissible sequence $\left<b_k\right>_1^\infty$. If $b_1 \le a_1-1$ then we use the identity \eqref{abs_rep} as well as condition--(i) to obtain the inequality
\[0 < \sum_{k=1}^\ell{b_k}|\theta_{k-1}| \le \sum_{k=1}^\infty{b_k}|\theta_{k-1}| \le (a_1-1)\theta_0 + \sum_{k=2}^\infty{b_k}|\theta_{k-1}| < (a_1-1)\theta_0 + \sum_{k=2}^\infty{a_k}|\theta_{k-1}|\]
\[ = (a_1-1)\theta_0 - a_1\theta_0 + \sum_{k=1}^\infty{a_k}|\theta_{k-1}| = (a_1-1)\theta_0 - a_1\theta_0+(1+\alpha) = 1.\]
If $b_1=a_1$, then by condition--(ii), we must have $b_2=0$. Let $n \ge 1$ be the first index for which $b_{2n+1} \le a_{2n+1} -1$, so that $b_{2k-1} = a_{2k-1}$ and $b_{2k} = 0$ for all $1 \le k \le n$ (the existence of $n$ is guaranteed by condition--(iii)). Using the recursive equation \eqref{theta_recursion}, we evaluate the finite sum
\[\sum_{k=1}^{2n}{a_k}|\theta_{k-1}| = \sum_{k=1}^{2n}{a_{2k-1}}|\theta_{2k-2}| = a_1|\theta_0| + a_3|\theta_2| + ... + a_{2n-1}|\theta_{2n-2}|\]  
\[ = (|\theta_{-1}| - |\theta_1|) + (|\theta_1| - |\theta_3|) + ... + (|\theta_{2n-1}| - |\theta_{2n+1}|) = |\theta_{-1}| - |\theta_{2n-1}| = 1 - |\theta_{2n+1}|.\]
In tandem with formula \eqref{eta_tail}, we obtain the desired inequalities 
\[0 < \sum_{k=1}^\infty{b_k}|\theta_{k-1}| = \sum_{k=1}^{2n}{a_{2k-1}}|\theta_{2k-2}| + b_{2n+1}|\theta_{2n}| + \sum_{k=2n+2}^\infty{b_k}|\theta_{k-1}|\]
\[ \le \sum_{k=0}^{n}{a_{2k-1}}|\theta_{2k-2}| + (a_{2n+1}-1)|\theta_{2n}| + \sum_{k=2n+2}^\infty{b_k}|\theta_{k-1}|\] 
\[< \sum_{k=0}^{n}{a_{2k-1}}|\theta_{2k-2}| + a_{2n+1}|\theta_{2n}| - |\theta_{2n}| + \sum_{k=2n+2}^\infty{a_k}|\theta_{k-1}|\] 
\[= (1 - |\theta_{2n-1}|) + (|\theta_{2n-1}| - |\theta_{2n+1}|) - |\theta_{2n}| + (|\theta_{2n}| + |\theta_{2n+1}|)=1\]
and conclude that $\sum_{k=1}^\ell{b_k}|\theta_{k-1}| \in (0,1)$. \\

\noindent Given $\beta \in (0,1)$ we obtain the limit $\ell$ and the sequence $\left<b_k\right>_1^\ell$ using the following iteration scheme:\\

\IncMargin{1em}
\begin{algorithm}[H]
\SetKwInOut{Input}{input}\SetKwInOut{Output}{output}
\Input{the base $\alpha \in(0,1)\backslash\IQ, \tab$ the initial seed $\beta \in (0,1)$}
\Output{the limit $\ell \in \IZ_{\ge 0}^\infty$, the $\alpha$--admissible digit sequence $\left<b_k\right>_1^\infty$}
\BlankLine
use algorithm \ref{rcf} and formula \eqref{theta_recursion} to obtain the sequence $\left<|\theta_k|\right>_0^\infty$\; 
set $\beta_0 := \beta, \ell := \infty,\tab k=1$\;
\While{$\beta_{k-1} > 0$}
{
set $b_k := \lfloor \beta_{k-1}/|\theta_{k-1}|\rfloor$\;
set $\beta_k := \beta_{k-1} - b_k|\theta_{k-1}|$\;
set $k:=k+1$\;
}
set $\ell := k-1$\;
\caption{Absolute Ostrowski Expansion}\label{ost_abs_inf}
\end{algorithm}\DecMargin{1em}
\vspace{1pc}

\noindent Since $\beta_0>0$, we must have $b_k\ge 1$ at least once. The assignment of $b_k$ in line--4 yields the inequality
\begin{equation}\label{b_k}
b_k|\theta_{k-1}| \le \beta_{k-1} < (b_k+1)|\theta_{k-1}|, \ttab 1 \le k < \ell+1.
\end{equation}
By the assignments of line--4 and line--5, we have 
\[\dfrac{\beta_{k-1}}{|\theta_{k-1}|}= b_k + \dfrac{\beta_{k}}{|\theta_{k-1}|} = \left\lfloor\dfrac{\beta_{k-1}}{|\theta_{k-1}|}\right\rfloor + \dfrac{\beta_{k}}{|\theta_{k-1}|},  \ttab 1 \le k < \ell+1,\]
that is, $b_k$ and $\beta_k$ are the quotient and remainder of the division of $\beta_{k-1}$ by $|\theta_{k-1}|$, hence $\beta_k < |\theta_{k-1}|$. This inequality and the inequalities \eqref{theta_a_k} and \eqref{b_k} imply that
\[b_k|\theta_{k-1}| \le \beta_{k-1} < |\theta_{k-2}| < (a_k+1)|\theta_{k-1}|,  \ttab 1 \le k < \ell+1.\]
Then for all $k$ we have $0 \le b_k \le a_k$ and, since the sequence $\left<|\theta_k|\right>_0^\infty$ is strictly decreasing to zero, we must also have $b_k\ge 1$ at least once, thus satisfying condition--(i). A simple comparison of the the sum \[\beta = \beta_0 = b_1|\theta_0| + \beta_1 = b_1|\theta_0| + b_2|\theta_1| + \beta_2 = ... = \sum_{k=1}^\ell{b_k}|\theta_{k-1}|\]
to the convergent series in proposition \ref{conv}, establishes its convergence and confirms that $\beta = \sum_{k=1}^\ell{b_k}|\theta_{k-1}|$. Furthermore, the Archemedean property of the field of real numbers asserts the uniqueness of the quotient $b_k$ and remainder $\beta_k$  in each iteration. Since this iteration terminates precisely when $\ell$ is finite, $\beta_{\ell-1}>0$ and $\beta_{\ell}=0$, the limit $\ell$ must be unique with $b_\ell \ge 1$ whenever it is finite.\\  

\noindent To establish condition--(ii), suppose $b_k=a_k$. Then we use the recursion formula \eqref{theta_recursion} and the iterative definitions of $b_k$ and $\beta_k$ in line--4 and line--5 to obtain the inequality
\[\beta_k = - b_k|\theta_{k-1}| + \beta_{k-1}= - a_k|\theta_{k-1}| + \beta_{k-1} =  |\theta_k| - |\theta_{k-2}| + \beta_{k-1}\]
\[=|\theta_k| - |\theta_{k-2}| - b_{k-1}|\theta_{k-2}| +  |\beta_{k-2}|< |\theta_k| - (b_{k-1}+1)|\theta_{k-2}| + (b_{k-1}+1)|\theta_{k-2}| = |\theta_k|.\]
Thus $\beta_k/|\theta_k| < 1$ so that in line--4 of the next iteration we must assign $b_{k+1} := 0$ as desired. \\

\noindent Finally, to establish condition--(iii), we assume by contradiction that $b_k \le a_k-1$ for only finitely many odd indexes $k$. Then we must have $\ell = \infty$ and there is some index $n \ge 1$ for which $b_{2k+1} = a_{2k+1}$ for all $k \ge n$. After we apply algorithm \ref{ost_abs_inf} to the inputs $\alpha := \alpha_{n-1}$ and $\beta := \beta_{n-1}$, we use the unitary representation \eqref{unit_rep} to arrive at the contradiction
\[1 = \sum_{k=0}^\infty{a_{2k+1}\theta_{2k}} = \sum_{k=0}^\infty{b_{2k+1}|\theta_{2k}|} \le \sum_{k:=1}^\infty{b_k|\theta_{k-1}|} = \beta < 1.\]
Thus $b_k \le a_k-1$ for infinitely many odd indexes $k$. To show this is true for infinitely many even indexes as well, we simply rewrite $\beta_0 := \beta_1$ so that all even indexes now become odd and repeat the previous argument. 
\end{proof}

\begin{corollary}\label{base_alpha}
Every real number $r$ can be uniquely expanded base--$\alpha$ as 
\[r = \displaystyle{\sum_{k=0}^\ell}b_k|\theta_{k-1}|,\] 
where $\ell \in \IZ_{\ge0}^\infty, b_0 \in \IZ$ and $\left<b_k\right>_1^\ell$ is an $\alpha$--admissible digit sequence.
\end{corollary} 

\begin{proof}
If $r$ is an integer we set $\ell:=0, \tab b_0:=r$ so that, by the definition of $\theta_{-1}=-1$ in the recursive formula \eqref{theta_recursion}, we obtain the vacuous expansion $r=b_0|\theta_{-1}|$. Furthermore, since $\sum_{k=1}^\ell{b_k}|\theta_{k-1}| \in (0,1)$ for all $\alpha$--admissible digit sequences $\left<b_k\right>_1^\infty$,  this expansion is unique. Otherwise, we set $b_0 := \lfloor r \rfloor$ and apply the theorem to $\beta_0 := r - b_0|\theta_{-1}| = r - \lfloor r \rfloor \in (0,1)$ and obtain the desired expansion. If $\left<b_k'\right>_0^{\ell'}$ is another $\alpha$--expansion for $r$, then $\sum_{k=1}^\ell{b_k}'|\theta_{k-1}| \in (0,1)$, hence we must have $b_0'=b_0=\lfloor r \rfloor$. The uniqueness of this expansion now guarantees that $\ell=\ell'$ and $b_k=b_k'$ for all $1 \le k < \ell$.
\end{proof}

\section{The Alternating Ostrowski Expansion}

\noindent The base--($-\alpha$) Alternating Ostowski Expansion is a sum of the form $\sum_{k:=1}^\infty{d_k}\theta_{k-1}$, where $0 \le d_k \le a_k$ along with all its finite truncations. As in the absolute case, uniqueness is not guaranteed. For instance,  after setting $\ell:=\infty, \tab c_1 := a_1-1,\tab c_{2k+1} := a_{2k+1}$ and $c_{2k}=0$ for all $k$, we use the definition \eqref{theta_def} of $\theta_0 := \alpha$ and the identity \eqref{unit_rep} to see that
\[\sum_{k=1}^\ell{c_k}\theta_{k-1} = \sum_{k=0}^\infty{c_{2k+1}}\theta_{2k} = \sum_{k=0}^\infty{a_{2k+1}}\theta_{2k} - \theta_0 = 1 -\alpha\]
is a different expansion than the one in the identity \eqref{cond_rep}. To achieve uniqueness, we will require the digit sequence to satisfy a refinement of the Markov Conditions. Given $\ell \in \IN^\infty$ and a sequence $\left<c_k\right>_1^\ell$, we say this sequence is $\boldsymbol{(-\alpha)}$\bf{--admissible} when it is $\alpha$--admissible and: 
\begin{enumerate}[label=(\roman*)]
\item if $c_{k+1} = 0$ then $c_k = a_k$ for all $1\le k < \ell-1$.
\item if $\ell=\infty$ then $c_k \ge 1$ for infinitely many odd and even indexes $k$.  
\end{enumerate}  
\begin{theorem}\label{Ost_alt}
For all $\ell \in \IZ_{\ge0}^\infty$ and $(-\alpha)$--admissible digit sequences $\left<c_k\right>_1^\infty$, we have\\ $\sum_{k=1}^\infty{c_k\theta_{k-1}} \in (-\alpha,1)$. Furthermore, for every real number $\gamma \in (-\alpha,1)$, there exists a unique limit $\ell \in \IZ_{\ge0}^\infty$ and a $(-\alpha)$--admissible digit sequence $\left<c_k\right>_1^\ell$ (with $c_\ell \ge 1$ when $\ell$ is finite) such that $\gamma = \sum_{k=1}^\ell{c_k}\theta_{k-1}$. 
\end{theorem}

\begin{proof}
Given a limit $\ell \in \IN^\infty$ and a $(-\alpha)$--admissible digit sequence $\left<c_k\right>_1^\ell$, we use condition--(i), condition--(iii) and the identities \eqref{theta_alt}, \eqref{self_rep} and \eqref{unit_rep} to obtain the inequality
\[-\alpha = \sum_{k=1}^\infty{a_{2k}}\theta_{2k-1} < \sum_{k=1}^\ell{c_k}\theta_{k-1}\le \sum_{k=0}^\ell{c_{2k+1}}\theta_{2k} < \sum_{k=0}^\infty{a_{2k+1}}\theta_{2k}= 1\]
and assert that $\sum_{k=1}^\ell{c_k}\theta_{k-1} \in (-\alpha,1)$. \\

\noindent Define the ceiling $\lceil r \rceil$ of the real number $r$ to be the smallest integer larger than or equal to $r$. Given $\gamma \in (-\alpha,1)$, we obtain the index $\ell$ and the sequence $\left<c_k\right>_1^\ell$ using the following iteration scheme:\\

\IncMargin{1em}
\begin{algorithm}[H]
\SetKwInOut{Input}{input}\SetKwInOut{Output}{output}
\Input{the base $\alpha \in(0,1)\backslash\IQ, \tab$ the initial seed $\gamma \in (-\alpha,1)$}
\Output{the limit $\ell \in \IZ_{\ge 0}^\infty$, the $(-\alpha)$--admissible digit sequence $\left<c_k\right>_1^\ell$}
\BlankLine
use algorithm \ref{rcf} and formula \eqref{theta_recursion} to obtain the sequence $\left<\theta_k\right>_0^\infty$\; 
set $\gamma_0 := \gamma, \ell := \infty, k:=1$\;
\While{$\gamma_{k-1} \ne 0$}
{
set $c_k := \min\{\lceil \gamma_{k-1}/\theta_{k-1}\rceil, a_k\}$\;
set $\gamma_k := \gamma_{k-1} - c_k\theta_{k-1}$\;
set $k:=k+1$\;
}
set $\ell := k-1$\;
\caption{Alternating Ostrowski Expansion}\label{ost_alt_inf}
\end{algorithm}\DecMargin{1em}
\vspace{1pc}
\noindent This iteration may terminate with a positive finite value for $\ell$ or continue indefinitely in which case $\ell = \infty$. We define the \bf{parity} $\rho(k)$ of $k$ to be one (zero) precisely when $k$ is odd (even), that is, $\rho(k) := \lceil k/2 \rceil - \lfloor k/2 \rfloor$ . We will first prove by induction that
\begin{equation}\label{gamma_k_induction}
\gamma_k \in \left(-\theta_{k-\rho(k)}, -\theta_{k-1+\rho(k)}\right), \ttab 0 \le k < \ell.
\end{equation}
By the definitions \eqref{theta_recursion} of $\theta_{-1}=-1$ and $\theta_0=\alpha$, the definition of $\gamma$ in the hypothesis and the assignment of line--2, we have $\gamma_0 = \gamma \in (-\alpha,1) = (-\theta_0,-\theta_{-1})$, hence the base case $k=0$ holds. After we assume its validity for $k-1$, we prove it is also true for $k$ by considering the two cases $\rho(k) \in \{0,1\}$ separately.\\

\noindent $\bullet$ If $\rho(k)=0$, then, by the induction assumption, we have $-\theta_{k-2}<\gamma_{k-1} < -\theta_{k-1}$. If in line--4, we set $c_k  =\lceil\gamma_{k-1}/\theta_{k-1}\rceil$, then from from formula \eqref{theta_alt} and the assignments of  line--5 we obtain
\[\dfrac{\gamma_k}{|\theta_{k-1}|}= -\dfrac{\gamma_k}{\theta_{k-1}} = c_k - \dfrac{\gamma_{k-1}}{\theta_{k-1}} = \left\lceil\dfrac{\gamma_{k-1}}{\theta_{k-1}}\right\rceil - \dfrac{\gamma_{k-1}}{\theta_{k-1}} \ge 0.\]
hence $\gamma_k \ge 0$. We use this inequality to obtain
\[-1< \dfrac{\gamma_{k-1}}{\theta_{k-1}} - \left\lceil\dfrac{\gamma_{k-1}}{\theta_{k-1}}\right\rceil = \dfrac{\gamma_{k-1}}{\theta_{k-1}} - c_k = \dfrac{\gamma_k}{\theta_{k-1}} = -\dfrac{\gamma_k}{|\theta_{k-1}|}\]
and conclude that $0 \le \gamma_k < -\theta_{k-1}$. If $c_k=a_k \le \lceil\gamma_{k-1}/\theta_{k-1}\rceil-1$, then
\[ \dfrac{\gamma_k}{|\theta_{k-1}|} = a_k-\dfrac{\gamma_{k-1}}{\theta_{k-1}} \le \left\lceil\dfrac{\gamma_{k-1}}{\theta_{k-1}}\right\rceil - 1 - \dfrac{\gamma_{k-1}}{\theta_{k-1}}  \le 0,\]
hence $\gamma_k\le0$. The recursion formula \eqref{theta_recursion}, the assignment of line--5 and the induction assumption will now yield 
\[0 \le -\gamma_k = a_k\theta_{k-1}-\gamma_{k-1} < a_k\theta_{k-1}+\theta_{k-2} = \theta_k.\]
Conclude that $-\theta_k<\gamma_k<-\theta_{k-1}$, which is the desired statement for the even index $k$.\\

\noindent $\bullet$ If $\rho(k)=1$, then, by the induction assumption, we have $-\theta_{k-1}<\gamma_{k-1} < -\theta_{k-2}$. If in line--4, we set $c_k  =\lceil\gamma_{k-1}/\theta_{k-1}\rceil$, then from formula \eqref{theta_alt} and the assignments of  line--5 we obtain
\[\dfrac{\gamma_k}{|\theta_{k-1}|}= \dfrac{\gamma_k}{\theta_{k-1}} = \dfrac{\gamma_{k-1}}{\theta_{k-1}} - c_k = \dfrac{\gamma_{k-1}}{\theta_{k-1}} - \left\lceil\dfrac{\gamma_{k-1}}{\theta_{k-1}} \right\rceil \le0,\]
hence $\gamma_k \le 0$. We use this inequality to obtain
\[-1< \dfrac{\gamma_{k-1}}{\theta_{k-1}} - \left\lceil\dfrac{\gamma_{k-1}}{\theta_{k-1}}\right\rceil = \dfrac{\gamma_{k-1}}{\theta_{k-1}} - c_k = \dfrac{\gamma_k}{\theta_{k-1}} = -\dfrac{\gamma_k}{|\theta_{k-1}|}\]
and conclude that $-\theta_{k-1}<\gamma_k\le0$. If $c_k=a_k \le \lceil\gamma_{k-1}/\theta_{k-1}\rceil-1$, then
\[ \dfrac{\gamma_k}{|\theta_{k-1}|} = \dfrac{\gamma_{k-1}}{\theta_{k-1}} - a_k \ge \dfrac{\gamma_{k-1}}{\theta_{k-1}}- \left\lceil\dfrac{\gamma_{k-1}}{\theta_{k-1}}\right\rceil + 1 \ge 0,\]
hence $\gamma_k\ge0$. The recursion formula \eqref{theta_recursion}, the assignment of line--5 and the induction assumption will now yield 
\[0 \le \gamma_k = \gamma_{k-1}-a_k\theta_{k-1} <-\theta_{k-2} - a_k\theta_{k-1} = -\theta_k.\]
Conclude that $-\theta_{k-1}<\gamma_k<-\theta_k$ for this case, which is the desired statement for the odd index $k$. This concludes the proof and asserts the validity of formula \eqref{gamma_k_induction}.\\ 

\noindent To establish condition--(i), assume that $\ell > 0$. Clearly by its definition in line--4, we have $c_k \le a_k$. Furthermore, by formula \eqref{gamma_k_induction} we either have have $-|\theta_k| = -\theta_k < \gamma_k$ when $\rho(k)=0$ or $\gamma_k < -\theta_k = -|\theta_k|$ when $\rho(k)=1$. In either case we see that $\gamma_k/\theta_k>-1$, so by the definition of $c_k$ in line--4 we conclude that $0 \le c_k \le a_k$ for all $k$ as desired. A simple comparison of the absolute terms in the the sum 
\[\gamma = \gamma_0 = c_1\theta_0 + \gamma_1 = c_1\theta_0 + c_2\theta_1 + \gamma_2 = ... = \sum_{k:=1}^\ell{c_k}\theta_{k-1},\]
to the convergent series in proposition \ref{conv}, establishes its convergence and confirms that $\gamma = \sum_{k=1}^\ell{c_k}\theta_{k-1}$. \\

\noindent To prove uniqueness, we split $\gamma$ into its positive and negative parts and invoke the uniqueness of the Absolute Ostrowski Expansion. More precisely, suppose $\left<c_k\right>_1^{\ell}$ is a $(-\alpha)$--admissible sequence such that $\gamma = \sum_{k=1}^{\ell}{c_k}\theta_{k-1}$. We first pad this sequence with an infinite tail of zeros whenever $\ell$ is finite and then define the terms
\[b_k^0 := \begin{cases} c_{k/2}, & \rho(k)=0\\ 0, & \rho(k)=1\end{cases}, \ttab b_k^1 := \begin{cases} 0, & \rho(k)=0\\ c_{(k+1)/2}, & \rho(k)=1 \end{cases}\] 
and the factors
\[\gamma^+ := \sum_{k=0}^{\infty}{c_{2k+1}}\theta_{2k} = \sum_{k=1}^{\infty}{b^1_k}|\theta_{k-1}|, \ttab \gamma^- := -\sum_{k=1}^{\lceil\ell/2\rceil}{c_{2k}}\theta_{2k-1} = \sum_{k=1}^{\infty}{b^0_k}|\theta_{k-1}|,\]
so that $\gamma = \gamma^+ - \gamma^-$. If $\left<c'_k\right>_1^{\ell'}$ is another $(-\alpha)$--admissible sequence such that\\ $\gamma = \sum_{k=1}^{\ell'}{c'_k}\theta_{k-1}$, then, we also pad it with a tail of zeros when applicable. Since both the sequences $\left<b_k^0\right>_1^\infty$ and $\left<b_k^1\right>_1^\infty$ are $\alpha$--admissible, the uniqueness of the absolute expansion implies that
\[\gamma^+ = \sum_{k=0}^{\infty}{c_{2k+1}}|\theta_{2k}| = \sum_{k=1}^{\infty}b_k^1|\theta_{k-1}| = \sum_{k=0}^{\infty}{c'_{2k+1}}|\theta_{2k}|\]
and
\[\gamma^- = \sum_{k=0}^{\infty}{c_{2k}}|\theta_{2k-1}| = \sum_{k=1}^{\infty}b_k^0|\theta_{k-1}| = \sum_{k=0}^{\infty}{c'_{2k}}|\theta_{2k-1}|,\]
hence $\left<c_k\right>_1^\infty = \left<c_k'\right>_1^\infty$. Furthermore, we must have $\ell=\ell'$ for otherwise we will obtain two distinct representations for either
\[\gamma^+ = \sum_{k=0}^{\lceil \ell/2 \rceil}{c_{2k+1}}\theta_{2k} = \sum_{k=0}^{\lceil \ell'/2 \rceil}{c_{2k+1}}\theta_{2k} \ttab \text{or} \ttab \gamma^- = \sum_{k=1}^{\lfloor\ell/2\rfloor}{c_{2k}}|\theta_{2k-1}| = \sum_{k=1}^{\lfloor\ell'/2\rfloor}{c_{2k}}|\theta_{2k-1}|,\]
contrary to the uniqueness of the absolute expansion. Conclude that this alternating expansion is also unique.\\

\noindent To establish condition--(ii), suppose $c_{k+1} = 0$. Then by it assignment in line--4, we must have 
\[\dfrac{\gamma_k}{\theta_k} < \left\lceil \dfrac{\gamma_k}{\theta_k} \right\rceil = c_{k+1} = 0\]
so that, using formulas \eqref{theta_alt}, we obtain that $\gamma_k>0$ precisely when $k$ is odd. Since $\theta_{k-1}>0$ precisely when $k$ is even, by the assignment of line--5, we will have
\[c_k = \dfrac{c_k\theta_{k-1}}{\theta_{k-1}} < \dfrac{\gamma_k+c_k\theta_{k-1}}{\theta_{k-1}} = \dfrac{\gamma_{k-1}}{\theta_{k-1}} \le \left\lceil\dfrac{\gamma_{k-1}}{\theta_{k-1}}\right\rceil.\]
Therefore, by its definition in line--4, we conclude that $c_k=a_k$. Finally, to establish condition--(iii), we assume by contradiction that $c_k \ge 1$ for only finitely many odd indexes $k$. Then we must have $\ell = \infty$ and there is some index $n \ge 1$ for which $c_{2k+1} = 0$ for all $k \ge n$. Then by condition-(ii) we must have $c_{2k}=a_{2k}$ for all $ k \ge n$. After we apply algorithm \ref{ost_alt_inf} to the inputs $\alpha := \alpha_{n-1}$ and $\gamma := \gamma_{n-1}$, we use the self representation \eqref{self_rep} to arrive at the contradiction
\[-\alpha < \gamma = \sum_{k=1}^\infty{c_k\theta_{k-1}} = \sum_{k=0}^\infty{c_{2k+1}\theta_{2k}} + \sum_{k=1}^\infty{c_{2k}}\theta_{2k-1} = \sum_{k=1}^\infty{a_{2k}}\theta_{2k-1} = -\alpha.\]
If $c_k \ge 1$ for only finitely many even indexes $k$, then we must have $\ell = \infty$ and there is some index $n \ge 1$ for which $c_{2k} = 0$ for all $k \ge n$. Then by condition-(ii) we must have $c_{2k-1}=a_{2k-1}$ for all $ k \ge n$. After we apply algorithm \ref{ost_alt_inf} to the inputs $\alpha := \alpha_{n-1}$ and $\gamma := \gamma_{n-1}$, we use the unitary representation \eqref{unit_rep} to arrive at the contradiction
\[1 = \sum_{k=0}^\infty{a_{2k+1}\theta_{2k}} = \sum_{k=0}^\infty{c_{2k+1}\theta_{2k}} = \sum_{k=0}^\infty{c_{2k+1}\theta_{2k}} + \sum_{k=1}^\infty{c_{2k}}\theta_{2k-1} = \sum_{k=1}^\infty{c_k\theta_{k-1}} = \gamma < 1.\]
\end{proof}

\begin{corollary}
Every real number $r$ can be uniquely expanded base--$(-\alpha)$ as 
\[r = \displaystyle{\sum_{k=0}^\ell}c_k\theta_{k-1},\] 
where $\ell \in \IZ_{\ge0}^\infty, \tab c_0 \in \IZ$ and $\left<c_k\right>_1^\ell$ is a $(-\alpha)$--admissible digit sequence with $c_1 \ge 1$.
\end{corollary}

\begin{proof}
If $r$ is an integer we set $\ell:=0, \tab c_0:=-r$ and use the definition of $\theta_{-1}:=-1$ in the recursive formula \eqref{theta_recursion} to obtain the unique vacuous expansion $r=c_0\theta_{-1}$. Otherwise, we set $c_0 := -\lfloor r \rfloor$ and apply the theorem to $\gamma_0 := r - c_0\theta_{-1} = r - \lfloor r \rfloor \in (0,1)$. Since $\gamma_0/\theta_0 >  0$, by its definition in line--4 of algorithm \ref{ost_alt_inf} we set $c_1 \ge 1$  and derive the desired expansion. If $\left<c_k'\right>_1^{\ell'}$ is another $(-\alpha)$--admissible with $c_1' \ge 1$ then so is the sequence obtained from the concatenation of $\left<c_1'-1\right>$ with $\left<c_k'\right>_2^{\ell'}$. If we are further supplied with an integer $c_0'$ such that $r=\sum_{k=0}^\ell{c'_k}\theta_{k-1}$ , then by the theorem we have 
\[-\alpha = -\theta_0 < (c_1'-1)\theta_0 + \sum_{k=2}^\ell{c_k}'\theta_{k-1} <1,\] 
hence $\sum_{k=0}^{\ell'}c_k'\theta_{k-1} \in (0,1)$. Thus we must have $c_0'=c_0=-\lfloor r \rfloor$ and then the uniqueness for this expansion guarantees that $\ell=\ell'$ and $c_k=c_k'$ for all $1 \le k < \ell$. 
\end{proof}

\section{Acknowledgments}

\noindent This work could have not been completed without the ongoing guidance, encouragement and good company of Robbie Robinson from George Washington University. 


\end{document}